\documentclass[12pt,reqno]{amsart}
\usepackage{enumerate, latexsym, amsmath, amsfonts, amssymb, amsthm, graphicx, color}
\def\pmod #1{\ ({\rm{mod}}\ #1)}
\def\Z{\Bbb Z}
\def\N{\Bbb N}

\def\Q{\Bbb Q}

\def\R{\Bbb R}

\def\l{\left}

\def\bg{\bigg}
\def\({\bg(}
\def\){\bg)}
\def\t{\text}

\def\ord{{\rm ord}}

\def\gen{{\rm gen}}
\def\cls{{\rm cls}}
\def\spn{{\rm spn}}

\def\ve{\varepsilon}

\theoremstyle{plain}
\newtheorem{theorem}{Theorem}

\newtheorem{lemma}{Lemma}
\newtheorem{corollary}{Corollary}
\newtheorem{conjecture}{Conjecture}

\theoremstyle{definition}
\newtheorem*{acknowledgment}{Acknowledgment}
\theoremstyle{remark}

\makeatletter
\@namedef{subjclassname@2010}{%
  \textup{2010} Mathematics Subject Classification}
\makeatother
 \vspace{4mm}

\begin{document}
 \baselineskip=17pt
\hbox{}
\medskip
\title[On the almost universality of $\lfloor x^2/a\rfloor+\lfloor y^2/b\rfloor+\lfloor z^2/c\rfloor$]
{On the almost universality of $\lfloor x^2/a\rfloor+\lfloor y^2/b\rfloor+\lfloor z^2/c\rfloor$}
\date{}
\author[Hai-Liang Wu, He-Xia Ni, Hao Pan] {Hai-Liang Wu, He-Xia Ni, Hao Pan}

\thanks{
\newline \indent Supported by the National Natural Science
Foundation of China (Grant No. 11571162).}

\address {(Hai-Liang Wu)  Department of Mathematics, Nanjing
University, Nanjing 210093, People's Republic of China}
\email{\tt whl.math@smail.nju.edu.cn}

\address{(He-Xia Ni)  Department of Mathematics, Nanjing
University, Nanjing 210093, People's Republic of China}
\email{\tt nihexia@yeah.net}

\address {(Hao Pan) Department of Mathematics, Nanjing
University of Finance $\&$ Economics, Nanjing 210023, People's Republic of China}
\email{{\tt haopan79@zoho.com}}

\begin{abstract}
In 2013, Farhi conjectured that for each $m\geq 3$, every natural number $n$ can be represented as $\lfloor x^2/m\rfloor+\lfloor y^2/m\rfloor+\lfloor z^2/m\rfloor$ with $x,y,z\in\Z$,
where $\lfloor\cdot\rfloor$ denotes the floor function.
Moreover, in 2015, Sun conjectured that every natural number $n$ can be written as $\lfloor x^2/a\rfloor+\lfloor y^2/b\rfloor+\lfloor z^2/c\rfloor$ with $x,y,z\in\Z$, where $a,b,c$ are integers and $(a,b,c)\neq (1,1,1),(2,2,2)$.

In this paper, with the help of congruence theta functions, we prove that for each $m\geq 3$, Farhi's conjecture is true for every sufficiently large integer $n$. And for  $a,b,c\geq 5$ with $a,b,c$ are pairwisely co-prime, we also confirm Sun's conjecture for every sufficiently large integer $n$.
\end{abstract}

\maketitle

\section{Introduction}
\setcounter{lemma}{0}
\setcounter{theorem}{0}
\setcounter{corollary}{0}
\setcounter{remark}{0}
\setcounter{equation}{0}
\setcounter{conjecture}{0}
\setcounter{proposition}{0}

The well-known Legendre three-square theorem says that any integer $n\in\N=\{0,1,2,\ldots\}$ can be represented as the sum of three squares (i.e., $n=x^2+y^2+z^2$ with $x,y,z\in\Z$), if and only if $n$ is not of the form $4^k(8h+7)$ where $k,h\in\N$. The
three-square theorem also implied Gauss' famous Eureka theorem: every natural number can be written as the sum of three triangular numbers, where the triangular numbers are the integers of the form $x(x+1)/2$ with $x\in\Z$. Nowdays, the arithmetical properties of the diagonal ternary quadratic form $ax^2+by^2+cz^2$, where $a,b,c$ are positive integers, have been deeply investigated. For example, in \cite[pp. 112-113]{D}, Dickson listed $102$ regular diagonal quadratic forms. For example, $n\in\N$ can be written as $2x^2+3y^2+3z^2$, if and only if $n$ is not of the form $9^k(3h+1)$. Another important example is Ramanujan's ternary quadratic form $x^2+y^2+10z^2$. Under the assumption of the generalized Riemann hypothesis, Ono and Soundararajan \cite{OS} proved that
$x^2+y^2+10z^2$ can represent every odd integer greater than $2719$.

On the other hand, in \cite{F13},  Farhi  considered the ternary sum of the form
$$
\bigg\lfloor \frac{x^2}m\bigg\rfloor+\bigg\lfloor \frac{y^2}m\bigg\rfloor+\bigg\lfloor \frac{z^2}m\bigg\rfloor,
$$ where $m$ is a positive integer and $\lfloor\cdot\rfloor$ is the floor function given by $\lfloor\theta\rfloor:=\max\{k\leq\theta:\, k\in\Z\}$.
Motived by the three-square theorem and the Eureka theorem,
Farhi proposed the following conjecture:
\begin{conjecture}
For each integer $m\geq 3$, every natural number $n$ can be represented as $n=\lfloor x^2/m\rfloor+\lfloor y^2/m\rfloor+\lfloor z^2/m\rfloor$ with $x,y,z\in\Z$.
\end{conjecture}
Note that the assumption $m\geq 3$ is necessary, since $\lfloor x^2/2\rfloor$ is  even for any $x\in\Z$. Farhi's conjecture is only verified for some small values of $m$:
\begin{itemize}
\item
$m=4,8$ by Farhi \cite{F13};

\item $m=3$ by Mezroui, Azizi and Ziane \cite{F14,MAZ};

\item $m=7,9,20,24,40,104,120$ by Holdum, Klausen and Rasmussen \cite{HKR};

\item $m=5,6,15$ by Sun \cite{Sun15}.\end{itemize}

Furthermore, in \cite{Sun15}, Sun investigated the ternary sum
\begin{equation}\label{Sunabc}
\bigg\lfloor \frac{x^2}a\bigg\rfloor+\bigg\lfloor \frac{y^2}b\bigg\rfloor+\bigg\lfloor \frac{z^2}c\bigg\rfloor,
\end{equation}
where $a,b,c$ are positive integers. Clearly (\ref{Sunabc}) is an analogue of the diagonal ternary quadratic forms. Sun introduced the following stronger conjecture:
\begin{conjecture}\label{SunConj}
Every natural number $n$ can be written as $n=\lfloor x^2/a\rfloor+\lfloor y^2/b\rfloor+\lfloor z^2/c\rfloor$ with $x,y,z\in\Z$, unless $(a,b,c)=(1,1,1)$ or $(2,2,2)$.
\end{conjecture}
Evidently Sun's conjecture implies Farhi's conjecture.
In the same paper, Sun had confirmed Conjecture \ref{SunConj} for the following cases:
\begin{itemize}
\item
$(a,b,c)=(1,1,m)$, $m=2,3,4,5,6,9,21$;

\item
$(a,b,c)=(1,m,m)$, $m=5,6,15$.\end{itemize}
Note that for any integers $a,b,c\geq 1$, there are infinitely many positive integers not represented by $ax^2+by^2+cz^2$ with $x,y,z\in\Z$. So it would be a meaningful problem to find more integer-valued ternary functions $F(x,y,z)$ such that $F(x,y,z)$, $x,y,z\in\Z$, can represent all natural numbers.

For an integer-valued ternary function $F(x,y,z)$, we say that $F(x,y,z)$ is {\it universal}, provided that every $n\in\N$ can be represented as $n=F(x,y,z)$ with $x,y,z\in\Z$. Clearly Sun's conjecture is equivalent to that $\lfloor x^2/a\rfloor+\lfloor y^2/b\rfloor+\lfloor z^2/c\rfloor$ is universal for $(a,b,c)\neq (1,1,1),(2,2,2)$. Respectively, $F(x,y,z)$ is called {\it almost universal}, provided that $F(x,y,z)$, $x,y,z\in\Z$, can represent every sufficiently large integer.

Define
$$\mathcal{F}_{a,b,c}(x,y,z):= \bigg\lfloor \frac{x^2}a\bigg\rfloor+\bigg\lfloor \frac{y^2}b\bigg\rfloor+\bigg\lfloor \frac{z^2}c\bigg\rfloor,$$
In particular, when $a=b=c=m$, we abbreviate $\mathcal{F}_{m,m,m}(x,y,z)$ as $\mathcal{F}_m(x,y,z)$.
Clearly the larger $a$ is, the denser $\{\lfloor x^2/a\rfloor:\, x\in\Z\}$ is.
However, as we shall see soon, the floor function $\lfloor\cdot\rfloor$ makes the properties of $\mathcal{F}_{a,b,c}(x,y,z)$ not very good, provided one of $a,b,c$ is large. In fact, even we don't know how to prove that there exist infinitely many square-free $m$ such that $\mathcal{F}_m(x,y,z)$, or $\mathcal{F}_{1,1,m}(x,y,z)$, or $\mathcal{F}_{1,m,m}(x,y,z)$, is universal. On the other hand, in this paper, we shall show that for each $m\geq 3$, Farhi's conjecture is true for every sufficiently large integer $n$.
In this paper, we shall prove that
\begin{theorem}\label{Thm a=b=c=m}
For each integer $m\ge3$, $\mathcal{F}_m(x,y,z)$ is almost universal, i.e., every sufficiently large integer $n$ can be written as
$$n=\bigg\lfloor \frac{x^2}m\bigg\rfloor+\bigg\lfloor \frac{y^2}m\bigg\rfloor+\bigg\lfloor \frac{z^2}m\bigg\rfloor$$
with $x,y,z\in\Z$.
\end{theorem}
For a positive integer $a$, let ${\mathfrak s}(a)$ denote the squarefree part of $a$, i.e., the largest squarefree divisor of $a$. Writing $a={\mathfrak s}(a)t^2$, clearly we have
\begin{equation}\label{tx2}
\bigg\lfloor \frac{(tx)^2}a\bigg\rfloor=\bigg\lfloor \frac{x^2}{{\mathfrak s}(a)}\bigg\rfloor.
\end{equation}
Thus it follows from Theorem \ref{Thm a=b=c=m} that
\begin{corollary}\label{Cor a=b=c=m}
Let $a,b,c$ be positive integers with ${\mathfrak s}(a)={\mathfrak s}(b)={\mathfrak s}(c)\geq3$. Then
$\lfloor x^2/a\rfloor+\lfloor y^2/b\rfloor+\lfloor z^2/c\rfloor$ is almost universal.
\end{corollary}

Furthermore, for some $(a,b,c)$, we can prove that Sun's conjecture is true for every sufficiently large integer $n$.
\begin{theorem}\label{Thm a,b,c}
Suppose that the integers $a,b,c\geq 5$ are pairwisely coprime. Then the ternary sum
$$\bigg\lfloor \frac{x^2}a\bigg\rfloor+\bigg\lfloor \frac{y^2}b\bigg\rfloor+\bigg\lfloor \frac{z^2}c\bigg\rfloor$$
is almost universal.
\end{theorem}
Similarly, we also have the following consequence of Theorem \ref{Thm a,b,c}
\begin{corollary}\label{Cor a,b,c}
Suppose that $a,b,c$ are positive integers and  ${\mathfrak s}(a),{\mathfrak s}(b),{\mathfrak s}(c)\geq4$ are pairwisely coprime. Then
$\lfloor x^2/a\rfloor+\lfloor y^2/b\rfloor+\lfloor z^2/c\rfloor$ is almost universal.
\end{corollary}
One way to deal with $\lfloor x^2/a\rfloor$ is to find a suitable residue $\alpha\pmod{a}$. Notice that
$$
\bigg\lfloor \frac{(ax+\alpha)^2}a\bigg\rfloor=ax^2+2\alpha x+\bigg\lfloor \frac{\alpha^2}a\bigg\rfloor.
$$
So by choosing some residues $\alpha\pmod{a}$, $\beta\pmod{b}$ and $\gamma\pmod{c}$, the ternary sum $\lfloor x^2/a\rfloor+\lfloor y^2/b\rfloor+\lfloor z^2/c\rfloor$ can be reduced to inhomogeneous quadratic polynomial
$$
x(ax+2\alpha)+y(by+2\beta)+z(cz+2\gamma).
$$

The arithmetic theory of inhomogeneous quadratic polynomial is a classical topic in number theory and has been extensively investigated. There are two useful tools for the inhomogeneous quadratic polynomial. One is the theory of shifted lattice, which is developed by  van der Blij \cite{V} and M. Kneser \cite{Kn}. The other is the the theory of congruence theta functions, which was established by
Shimura \cite{Shimura1,Shimura2}. For the recent progress on the universality and almost universality of inhomogeneous quadratic polynomials, the reader may refer to
\cite{WKCOH,WKCANNA,GS,ANNA,ANNA2,KS,Sun15b,Sun17}.

Now, we give an outline of this paper. In Section 2, we will give a brief overview of the theory of shifted lattice and congruence theta functions. Meanwhile, we shall prove some useful lemmas. Finally, we shall prove
Theorem \ref{Thm a=b=c=m} and Theorem \ref{Thm a,b,c} in Section 3 and Section 4 respectively.
\maketitle
\section{Preliminaries}
\setcounter{lemma}{0}
\setcounter{theorem}{0}
\setcounter{corollary}{0}
\setcounter{remark}{0}
\setcounter{equation}{0}
\setcounter{conjecture}{0}

In this section, we first introduce some necessary objects used in our proofs and then clarify the connections
between congruence theta function and representation of natural numbers by $\mathcal{F}_{a,b,c}(x,y,z)$.
\subsection{Shifted Lattice Theory}
In this paper, we adopt the language of quadratic space and lattices as in \cite{Ki}. Given a
positive definite quadratic space $(V,B,Q)$ over $\Q$, let $L$ be a lattice on $V$ and $A$ be a symmetric matrix.
We denote $L\cong A$ if $A$ is the gram matrix for $L$ with respect to some basis. In particular,
an $n\times n$ diagonal matrix with $a_1,...,a_n$ as the diagonal entries is written as $\langle a_1,...,a_n\rangle$.
Let $\Omega$ denote the set of all places of $\Q$, for each place $p\in\Omega$, we define the localization of $L$ by
$L_p=L\otimes_{\Z}\Z_p$, especially, $L_{\infty}=V\otimes_{\Q}\R$, and the adelization of $O^+(V)$ is the set
$$O^+_{\mathbb{A}}(V):=\{(\sigma_v)\in\prod_{v\in\Omega}O^+(V_v): \sigma_p(L_p)=L_p\ \text{for almost all primes}\ p\}.$$
We also set
$$O'_{\mathbb{A}}(V):=\{(\sigma_v)\in O^+_{\mathbb{A}}(V): \sigma_v\in O'(V_v)\ \text{for each place}\ v\},$$
where $O'(V_v)$ denotes the kernel of the spinor norm from $O^+(V_v)$ to $\Q_v^{\times}/\Q_v^{\times2}$.
For more details of the adelization groups $O^+(V)$, $O'_{\mathbb{A}}(V)$, $O^+_{\mathbb{A}}(V)$, readers may consult \cite[p.132]{Ki}.

For a vector $v\in V$, we have a lattice coset $L+v$. Similar to the definitions of
$\cls(L),\ \gen(L),\ \spn(L)$ (cf. \cite[pp.132--133]{Ki}), we have the following definitions that
originally appear in \cite{V}.

The proper class of $L+v$ is defined as
\begin{align*}
\cls^+(L+v):= \text{the orbit of $L+v$ under the action of $O^+(V)$},
\end{align*}

the proper genus of $L+v$ is defined as
\begin{align*}
\gen^+(L+v):= \text{the orbit of $L+v$ under the action of $O^+_{\mathbb{A}}(V)$},
\end{align*}

the proper spinor genus of $L+v$ is defined as
\begin{align*}
\spn^+(L+v):= \text{the orbit of $L+v$ under the action of $O^+(V)O'_{\mathbb{A}}(V)$},
\end{align*}

we now consider the representation of a natural number $n$ by $\mathcal{F}_{a,b,c}(x,y,z)$. Put
\begin{equation}\label{delta}
\delta=\begin{cases}1&\mbox{if}\ 2\nmid (a,b,c),\\0&\mbox{otherwise}.\end{cases}
\end{equation}
Given any
$\alpha,\beta,\gamma\in\Z$, there exists a unique triple $(a_0,b_0,c_0)\in\Z^3$ satisfying the following conditions:
$0\le a_0<a$, $0\le b_0<b$, $0\le c_0<c$, $a_0\equiv \alpha^2\pmod a$, $b_0\equiv \beta^2\pmod b$,
$c_0\equiv \gamma^2\pmod c$. Then by virtue of the discussion under the corollary \ref{Cor a,b,c},
it is easy to see that $n$ can be represented by $\mathcal{F}_{a,b,c}$ if
and only if there are $\alpha,\beta,\gamma\in\Z$ such that $l_{a,b,c}(n)$ can be represented by the lattice
coset $L+v$, where
\begin{equation}\label{Lattice a,b,c}
L=4^{\delta}abc\langle a,b,c\rangle
\end{equation}
with respect to the orthogonal basis $\{e_1,e_2,e_3\}$,
\begin{equation}\label{vector a,b,c}
v=2^{-\delta}(\frac{\alpha}{a}e_1+\frac{\beta}{b}e_2+\frac{\gamma}{c}e_3),
\end{equation}
and
\begin{equation}\label{l_a,b,c,n}
l_{a,b,c}(n)=abcn+bca_0+cab_0+abc_0.
\end{equation}
In particular, when $a=b=c=m$, we can simplify our problem, and one may easily verify that
$n$ can be represented by $\mathcal{F}_m$ if
and only if there are $\alpha,\beta,\gamma\in\Z$ such that $l_m(n)$ can be represented by the lattice
coset $M+u$, where
\begin{equation}\label{Lattice a=b=c=m}
M=4^{\delta}m^2\langle 1,1,1\rangle
\end{equation}
with respect to the orthogonal basis $\{e_1,e_2,e_3\}$,
\begin{equation}\label{vector a=b=c=m}
u=\frac{1}{2^{\delta}m}(\alpha e_1+\beta e_2+\gamma e_3),
\end{equation}
and
\begin{equation}\label{l_m,n}
l_m(n)=mn+a_0+b_0+c_0.
\end{equation}

For each rational prime $p$, let $\Z_p$ denote the set of all $p$-adic integers and $\Z_p^{\times}$ denote the
set of all $p$-adic units. For the convenience of readers, we state the result of Local Square Theorem (cf. \cite[Lemma 1.6]{C}) here, and we will use
this result to prove Lemma \ref{Lem local}.

{\bf Local Square Theorem} For each prime $p$, let $\alpha\in\Q_p^{\times}$ and let $\ord_p$ denote the $p$-adic order.
Suppose that $\ord_p(\alpha-\ve^2)\ge\ord_p(4p)$
for some $\ve\in\Z_p^{\times}$. Then $\alpha$ is a square in $\Z_p^{\times}$.

We are now in a position to state the following lemma involving local representation.

\begin{lemma}\label{Lem local}
{\rm (i)} Given an odd prime $p$, for any $\ve$, $\eta\in\Z_p^{\times}$ and integer $k\ge1$, we have
$$\Z_p=\{p^k\ve x^2+\eta x: x\in\Z_p\}.$$

{\rm (ii)} Let $\ve$, $\eta\in\Z_2^{\times}$ and integer $k\ge2$, we have
$$4\Z_2\subseteq\{2\ve x^2+2\eta x: x\in\Z_2\}$$
and
$$2\Z_2\subseteq\{2^k\ve x^2+2\eta x: x\in\Z_2\}.$$
\end{lemma}
\begin{proof}
(i) By Hensel's Lemma, it is clear that $p^k\ve x^2+\eta x$ can represent all $p$-adic integers over $\Z_p$.

(ii) For each $n\in\Z_2$, the equation
\begin{equation}\label{equation A}
2\ve x^2+2\eta x=4n
\end{equation}
has a solution $x$ in the algebraic closure of $\Q_2$, where
\begin{equation*}
x=\frac{-\eta\pm\sqrt{\eta^2+8n\ve}}{2\ve}=\frac{-1\pm\sqrt{1+8n\ve/\eta^2}}{2\ve/\eta}.
\end{equation*}
By Local Square Theorem, we may write
$\eta^2+8n\ve=\eta^2(1+2^s\mu)^2$ with $s\ge 1$ and $\mu\in\Z_2^{\times}$. Thus
\begin{equation*}
x=\frac{-\eta+\eta(1+2^s\mu)}{2\ve}=2^{s-1}\mu\eta/\ve\in\Z_2
\end{equation*}
is a solution of (\ref{equation A}).
Now, we consider the equation
\begin{equation}\label{equation B}
2^k\ve x^2+2\eta x=2n.
\end{equation}
Then it has a solution $x$ in the algebraic closure of $\Q_2$, where
\begin{equation*}
x=\frac{-\eta\pm\sqrt{\eta^2+2^{k+1}n\ve}}{2^k\ve}=\frac{-1\pm\sqrt{1+2^{k+1}n\ve/\eta^2}}{2^k\ve/\eta}.
\end{equation*}
Since $k\ge2$,
by Local Square Theorem, we may write $1+2^{k+1}n\ve/\eta^2=(1+2^s\mu)^2$ with $\mu\in\Z_2^{\times}$ and $s\ge1$. If $s=1$, then we have $2^{k-1}n\ve/\eta^2=\mu(1+\mu)$, this implies that $\ord_2(1+\mu)\ge k-1$, thus
\begin{equation*}
x=\frac{-1-(1+2\mu)}{2^k\ve/\eta}=\frac{(1+\mu)\eta}{2^{k-1}\ve}\in\Z_2
\end{equation*}
is a solution of (\ref{equation B}).
If $s>1$, then we have $2^{k}n\ve/\eta^2=2^{s}\eta(1+2^{s-1}\eta)$, this implies that $s\ge k$, thus
\begin{equation*}
x=\frac{-1+(1+2^s\mu)}{2^k\ve/\eta}=2^{s-k}\mu\eta/\ve\in\Z_2
\end{equation*}
is a solution of (\ref{equation B}).

In view of the above, we complete the proof.
\end{proof}

\subsection{Congruence Theta Functions}
In \cite{Shimura1} Shimura investigated the modular forms of half-integral weights and explicit
transformation formulas for the theta functions. Let $n$ and $N$ be positive integers, let $A$ be an $n\times n$ integral positive definite symmetric matrix with $NA^{-1}$ be an integral matrix. Given a non-negative integer $l$, let $P$ be a spherical function of order $l$ with respect to $A$. For a column vector $h\in\Z^n$ with $Ah\in N\Z^n$, we consider the following theta function (\cite[2.0]{Shimura1}) with variable $z\in\mathbb{H}$ (where $\mathbb{H}$ denotes the upper half plane),
\begin{equation}\label{theta function}
\theta(z,h,A,N,P)=\sum_{m\equiv h\pmod N}P(m)\cdot e(z\cdot m^tAm/2N^2),
\end{equation}
where the summation runs over all $m\in\Z^n$ with $m\equiv h\pmod N$, and  
$m^t$ denotes the transposed matrix of $m$ and $e(z)=e^{2\pi iz}.$ We usually call the theta function of this type {\it congruence theta function}. When $P=1$,
we simply write $\theta(z,h,A,N)$ instead of $\theta(z,h,A,N,P)$.
By \cite[Propostion 2.1]{Shimura1}, if the diagonal elements of $A$ are even, then
$\theta(z,h,A,N,P)$ is a modular form of weight $n/2$ on $\Gamma_1(2N)$ with some multiplier.

Congruence theta functions have close connection with the representations of natural numbers by quadratic forms with some additional congruence conditions. To see this, we turn to the representation of natural number $n$ by $\mathcal{F}_{a,b,c}(x,y,z)$.
Let $L$ and $v$ be as in (\ref{Lattice a,b,c}) and (\ref{vector a,b,c}) respectively. We define
\begin{equation}\label{theta F}
\theta_{L+v}(z):=\sum_{n\ge0}r_{L+v}(n)e(nz),
\end{equation}
where $z\in\mathbb{H}$, and $r_{L+v}(n):=\#\{x\in L: Q(x+v)=n\}$ ($\#S$ denotes the cardinality of a finite set $S$).
On the other hand,
put
$$N=2^{\delta}abc,
A=2^{\delta}\langle a,b,c\rangle,
h=(bc\alpha,ca\beta,ab\gamma)^t,$$
where $\delta$ is as in (\ref{delta}).

In particular, when $a=b=c=m$, we simply set
$$N=2^{\delta}m, A=2^{\delta}\langle m,m,m\rangle, h=(\alpha,\beta,\gamma)^t.$$

Then one may easily verify that $\theta_{L+v}(z)=\theta(2Nz,h,A,N)$, thus $\theta_{L+v}(z)$ is a modular form
of weight $3/2$ on $\Gamma_1(4N^2)$. In addition, we need the following unary congruence theta functions,
let $N=2^{\delta}abc$ with $\delta$ as in (\ref{delta}), and let $N=2^{\delta}m$ if $a=b=c=m$. 
For any positive squarefree factor $t$ of $N$, we further set $A=(N/t)$ and $P(n)=n$. We now define
\begin{equation}\label{unary}
u_{h,t}(z):=\theta(2Nz,h,A,N/t,P)=\sum_{r\equiv h\pmod {N/t}}re(tr^2z),
\end{equation}
where $h$ be chosen modulo $N/t$ and the summation runs over all integers $r$ such that $r\equiv h\pmod {N/t}$. Then $u_{h,t}(z)$ is a modular form of weight $3/2$ on $\Gamma_1(4N^2)$ with the same multiplier
of $\theta_{L+v}(z)$.

As in \cite{WR,AK}, we may decompose $\theta_{L+v}(z)$ into the following three parts
\begin{equation}\label{decomposition}
\theta_{L+v}(z)=\mathcal{E}(z)+\mathcal{U}(z)+f(z),
\end{equation}
where $\mathcal{E}(z)$ is in the space generated by Eisenstein series, $\mathcal{U}(z)$ is in the space generated by unary theta functions defined in (\ref{unary}), and $f(z)$ is a cusp form which is orthogonal to those unary theta functions.

We first study the function $\mathcal{E}(z)$. Shimura \cite{Shimura2} proved that
$\mathcal{E}(z)$ is the weighted average of representations by the members of the genus of $L+v$, and simultaneously the product of local densities. Thus if $n$ can be represented by $L+v$ locally and has bounded divisibility at each anisotropic prime of $V$, then for any $\ve>0$, its $n$-th fourier coefficient $a_E(n)\gg n^{1/2-\ve}$. By virtue of \cite{Duke}, the $n$-th fourier coefficient of $f(z)$ grows at most like
$n^{3/7+\ve}$. Hence, let $\mathcal{S}\subseteq\N$ be an infinite set. For each $n\in \mathcal{S}$, suppose that  $l_{a,b,c}(n)$ is represented by $L+v$ locally and has bounded divisibility at each anisotropic prime of $V$, and that
the $l_{a,b,c}(n)$-th fourier coefficient of $\mathcal{U}$ is equal to zero. Then 
it is easy to see that every sufficiently large integer $n\in \mathcal{S}$ can be written as
$$\lfloor x^2/a\rfloor+\lfloor y^2/b\rfloor+\lfloor z^2/c\rfloor,$$
with $x,y,z\in\Z$.

Since $\mathcal{U}(z)$ is a linear combination of finitely many unary congruence theta functions defined in
(\ref{unary}), thus the possible $u_{h,t}(z)$ in the decomposition with non-zero coefficient must satisfy
\begin{equation}\label{tr^2}
\begin{cases}l_{a,b,c}(n)=tr^2\equiv0\pmod t,\\t\mid (l_{a,b,c}(n),\ 2^{\delta}abc).\end{cases}
\end{equation}
In particular, when $a=b=c=m$, the possible $u_{h,t}(z)$ in the decomposition with non-zero coefficient must satisfy
\begin{equation}\label{tr^2, m}
\begin{cases}l_{m}(n)=tr^2\equiv0\pmod t,\\t\mid (l_{m}(n),\ 2^{\delta}m).\end{cases}
\end{equation}

In view of the above, set
\begin{align*}
\mathcal{S}_{a,b,c}^t=&\{n\in\N: l_{a,b,c}(n)=tr^2\ \text{for some $r\in\N$}\},\\
\mathcal{S}_{a,b,c}=&\bigcup_{\substack{t\mid N\\ \text{t squarefree}}}\mathcal{S}_{a,b,c}^t.
\end{align*}
In particular, when $a=b=c=m$, we simply write $\mathcal{S}_m^t$ and $\mathcal{S}_m$ instead of $\mathcal{S}_{a,b,c}^t$ and $\mathcal{S}_{a,b,c}$ respectively.
We have the following lemma.
\begin{lemma}\label{Lem summary}
Let notations be as above and $\mathcal{S}\subseteq\N$ be an infinite set satisfying
$\#(\mathcal{S}\cap\mathcal{S}_{a,b,c})<\infty$. Suppose that for almost all $n\in \mathcal{S}$,
$l_{a,b,c}(n)$ can be represented by $L+v$ locally and has bounded divisibility at each anisotropic prime of $V$, then every sufficiently large integer $n\in\mathcal{S}$ can be written as
$$\lfloor x^2/a\rfloor+\lfloor y^2/b\rfloor+\lfloor z^2/c\rfloor,$$
with $x,y,z\in\Z$.
\end{lemma}

\maketitle
\section{Proof of Theorem \ref{Thm a=b=c=m}}
\setcounter{lemma}{0}
\setcounter{theorem}{0}
\setcounter{corollary}{0}
\setcounter{remark}{0}
\setcounter{equation}{0}
\setcounter{conjecture}{0}

In this Section, we adopt the notations as in Section 1 and Section 2. As in the discussion of Section 2, we need to find suitable
$\alpha,\beta,\gamma$ such that $M+u$ can represent almost all natural numbers.
We first need the following lemma involving local representation.
\begin{lemma}\label{Lem p>2}
{\rm (i)} For each odd prime $p$ not dividing $m$, then $l_m(n)$ is represented by $M_p+u$.

{\rm (ii)} For each odd prime $p$ dividing $m$, if at least one of $\alpha,\beta,\gamma$ is not divisible by $p$,
then $l_m(n)$ is represented by $M_p+u$.
\end{lemma}
\begin{proof}
(i) In this case, $M_p+u=M_p$, and $M_p$ is unimodular, thus
$$M_p\cong \begin{pmatrix} 0 & 1 \\1 & 0 \end{pmatrix}\perp\langle-d(M)\rangle,$$
where $d(M)$ denotes the discriminant of $M$, clearly, $M_p$ can represent all $p$-adic integers over $\Z_p$.

(ii) Without loss of generality, we may assume that $p\nmid \alpha$. By (i) of Lemma \ref{Lem local}, for each
$n\in\Z_p$,
$n+(a_0+b_0+c_0-\alpha^2-\beta^2-\gamma^2)/m$ is represented by $4^{\delta}mx^2+2^{\delta+1}\alpha x$, thus
$l_m(n)$ is represented by $(2^{\delta}mx+\alpha)^2+(2^{\delta}my+\beta)^2+(2^{\delta}mz+\gamma)^2$, i.e.,
$l_m(n)$ can be represented by $M_p+u$.
\end{proof}

\noindent{\it\bf Proof of Theorem 1.1}. We will divide the proof into following three cases.

{\it Case} 1. $\ord_2m=1$.

In this case, since Sun \cite{Sun15} proved that each natural number can be represented by $\mathcal{F}_6$,
we may assume that $m\ge10$.
We first consider the local representation of $l_m(n)$ by $M_2+u$. Via computation, it is easy to see that
$l_m(n)$ is represented by $M_2+u$ if and only if $l^*_m(n):=n+(a_0+b_0+c_0-\alpha^2-\beta^2-\gamma^2)/m$ is represented
by $f^*(x,y,z):=mx^2+2\alpha x+my^2+2\beta y+mz^2+2\gamma z$ over $\Z_2$.

When $n$ is even, we put $a_0=\alpha=1$, $b_0=\beta=c_0=\gamma=0$, then $l^*_m(n)=n$. By (ii) of Lemma \ref{Lem local},
$n-m\ve^2$ is represented by $mx^2+2\alpha x$ over $\Z_2$, i.e., $l^*_m(n)$ is represented by $f^*(x,y,z)$ over $\Z_2$,
where
$$\ve=\begin{cases}1&\mbox{if}\ n\equiv2\pmod4,\\0&\mbox{if}\ n\equiv0\pmod4.\end{cases}$$
Combining Lemma \ref{Lem p>2}, $l_m(n)=mn+1$ is represented by $M_p+u$ for each prime $p$ and $l_m(n)\equiv 1\pmod4$.

When $n$ is odd, for convenience, we set $m=2(2l+1)$ with $l\in\Z^+$. We further put
\begin{equation*}
k=\begin{cases}l&\mbox{if}\ m\equiv6\pmod8,\\l-1&\mbox{if}\ m\equiv2\pmod8,\end{cases}
\end{equation*}
\begin{equation*}
\gamma=\begin{cases}m/2&\mbox{if}\ m\equiv6\pmod8,\\m/2-1&\mbox{if}\ m\equiv2\pmod8,\end{cases}
\end{equation*}
and
\begin{equation*}
c_0=\gamma^2-km=\begin{cases}m/2&\mbox{if}\ m\equiv6\pmod8,\\m/2+1&\mbox{if}\ m\equiv2\pmod8.\end{cases}
\end{equation*}
Note that $k$ is odd and $0\le c_0<m$. Let $k,\gamma,c_0$ be as above, and put
$\alpha=\frac{3+(-1)^{(m+2)/4}}{2}$, $a_0=\alpha^2$, $\beta=b_0=0$, then $l^*_m(n)=n-k$. Since $k$ is odd, clearly, 
$n-k-m(\frac{1-(-1)^{(n-k)/2}}{2})^2\equiv0\pmod4$. Note that either $\alpha$ or $\gamma$ is odd and that
$\alpha\in\{1,2\}$, thus by (ii) of
Lemma \ref{Lem local} and Lemma \ref{Lem p>2}, $l_m(n)=mn+m/2+3+(-1)^{(m+2)/4}$ is represented by $M_p+u$ for each prime $p$ and $l_m(n)\equiv 1\pmod4$.

In view of the above, any odd prime divisor of $m$ is prime to $l_m(n)$. By (\ref{tr^2, m}) and the fact that $l_m(n)\equiv 1\pmod4$, we must have
$\mathcal{S}_m=\mathcal{S}_m^1$. If $n\not\in\mathcal{S}_m^1$, then by Lemma \ref{Lem summary}, each sufficiently large $n$ can be represented by $\mathcal{F}_m$. If $n\in\mathcal{S}_m^1$, that is,
$mn+a_0+b_0+c_0=x^2$ for some $x\in\Z$, by the above choice of $a_0,b_0,c_0$ and the assumption that $m\ge10$, we have $0\le a_0+b_0+c_0<m$, this implies that
$n=\lfloor x^2/m\rfloor+\lfloor 0^2/m\rfloor+\lfloor 0^2/m\rfloor$.
\medskip

{\it Case} 2. $\ord_2m\ge2$.

As in the Case 1, it is easy to see that
$l_m(n)$ is represented by $M_2+u$ if and only if $l^*_m(n)=n+(a_0+b_0+c_0-\alpha^2-\beta^2-\gamma^2)/m$ is represented
by $f^*(x,y,z)=mx^2+2\alpha x+my^2+2\beta y+mz^2+2\gamma z$ over $\Z_2$.

When $n$ is even, put $\alpha=a_0=1$, $\beta=\gamma=b_0=c_0=0$, then $l^*_m(n)=n$. By (ii) of Lemma \ref{Lem local} and Lemma \ref{Lem p>2}, $l_m(n)=mn+1$ is represented by $M_p+u$ for each prime $p$ and $l_m(n)\equiv1\pmod4$.

When $n$ is odd, we set
\begin{equation*}
k=\begin{cases}m/4&\mbox{if}\ \ord_2m=2,\\m/4-1&\mbox{otherwise},\ \end{cases}
\end{equation*}
\begin{equation*}
\gamma=\begin{cases}m/2&\mbox{if}\ \ord_2m=2,\\m/2-1&\mbox{otherwise},\ \end{cases}
\end{equation*}
\begin{equation*}
c_0=\gamma^2-km=\begin{cases}0&\mbox{if}\ \ord_2m=2,\\1&\mbox{otherwise},\ \end{cases}
\end{equation*}
and set 
$\alpha=a_0=1-c_0$.
Note that $k$ is odd and $0\le c_0<m$. Let $\alpha,\gamma,a_0,c_0$ be as above, and let $\beta=b_0=0$, then
$l^*_m(n)=n-k$. Since $n-k$ is even, and note that either $\alpha$ or $\gamma$ is odd and that
at least one of $\alpha,\gamma$ is prime to $m$. Thus by (ii) of Lemma \ref{Lem local} and Lemma \ref{Lem p>2}, $l_m(n)=mn+1$ is represented by $M_p+u$ for each prime $p$ and $l_m(n)\equiv1\pmod4$. Finally, with the essentially same method at the end of Case 1, one can easily get the
desired results.
\medskip

{\it Case} 3. $2\nmid m$.

In this case, since S. Mezroui, A. Azizi and M. Ziane \cite{MAZ} proved that $\mathcal{F}_3$ can represent each natural number, we may assume that $m\ge5$.
We put $e_i'=\frac{1}{2}e_i$ for $i=1,2,3$ and $M'=\Z e_1'\perp \Z e_2'\perp \Z e_3'$. If $l_m(n)$ is represented by $M'_2+u$, then there are $x_i\in\Z_2$ ($i=1,2,3$) such that
\begin{equation}\label{equation C}
l_m(n)=Q\left((x_1+\frac{\alpha}{m})e_1'+(x_2+\frac{\beta}{m})e_2'+(x_3+\frac{\gamma}{m})e_3'\right).
\end{equation}
For $i=1,2,3$, let
$$\ve_i=\begin{cases}1&\mbox{if}\ 2\nmid x_i,\\0&\mbox{if}\ 2\mid x_i,\end{cases}$$
and let $\alpha'=\alpha+\ve_1m$, $\beta'=\beta+\ve_2m$, $\gamma'=\gamma+\ve_3m$, then (\ref{equation C})
is equal to
\begin{equation*}
Q\left(\frac{x_1-\ve_1}{2}e_1+\frac{x_2-\ve_2}{2}e_2+\frac{x_3-\ve_3}{2}e_3+
\frac{1}{2m}(\alpha'e_1+\beta'e_2+\gamma'e_3)\right).
\end{equation*}
If we replace $\alpha,\beta,\gamma$ by $\alpha',\beta',\gamma'$ respectively, then
it might worth mentioning here that $l_m(n)$ is invariant and the above results on the local representation of $l_m(n)$ over $\Z_p$ ($p\ne2$) still hold.
Thus we only need to consider the representation of $l_m(n)$ by $M_2'+u$ over $\Z_2$.
Since $2\nmid m$ and
$u=\frac{1}{m}(\alpha e_1'+\beta e_2'+\gamma e_3')$, thus $M_2'+u=M_2'$. As $M_2'$ is anisotropic,
by \cite[Lemma 5.2.6 and Propositon 5.2.3]{Ki},
we have
$$M_2'\cong \begin{pmatrix} 2 & 1 \\1 & 2 \end{pmatrix}\perp \langle 3\rangle,$$
and \begin{equation}\label{M2}
Q(\begin{pmatrix} 2 & 1 \\1 & 2 \end{pmatrix})=\{x\in\Z_2: \ord_2x\equiv1\pmod2\}\cup\{0\}
\end{equation}
(for a $\Z_2$-lattice $K$, $Q(K)$ denotes the set $\{x\in K: Q(x)\}$).

When $n\equiv 0\pmod4$, put $\alpha=a_0=1$, and $\beta=b_0=\gamma=c_0=0$, then
$l_m(n)=mn+1\equiv 1\pmod 4$. By (\ref{M2}), it is easy to see that $l_m(n)$ is represented by $M_2'$.

When $n\equiv 2\pmod4$, we put $\alpha=\beta=\gamma=a_0=b_0=c_0=1$, then
$l_m(n)=mn+3\equiv 1\pmod4$. By (\ref{M2}), $l_m(n)$ is represented by $M_2'$.
It needs noting here that $mn+3$ is not of the form $3r^2$ with $r\in\N$, since if $mn+3=3r^2$ for some $r\in\N$, then $mn/3+1=r^2\equiv 3\pmod4$, a contradiction.

When $n$ is odd, let $\gamma=c_0=0$, $a_0=\alpha^2$ and $b_0=\beta^2$, where
$$(\alpha,\beta)=\begin{cases}(2,0)&\mbox{if}\ m\equiv n\pmod4,\\(1,1)&\mbox{otherwise}.\ \end{cases}$$
Thus $l_m(n)=mn+a_0+b_0\equiv 1\pmod4$. By (\ref{M2}), it is easy to see that
$l_m(n)$ is represented by $M'_2$.

In sum, by Lemma \ref{Lem p>2} and (\ref{tr^2, m}), $l_m(n)$ can be represented by $M'+u$ locally.
This implies that $l_m(n)$ can be represented by $M+u'$ locally, where
$u'=\frac{1}{2m}(\alpha'e_1+\beta'e_2+\gamma'e_3)$. It is easy to verify that
$\mathcal{S}_m=\mathcal{S}_m^1$. If $n\not\in\mathcal{S}_m^1$, then by Lemma \ref{Lem summary}, each sufficiently large $n$ can be represented by $\mathcal{F}_m$. If $n\in\mathcal{S}_m^1$, that is,
$mn+a_0+b_0+c_0=x^2$ for some $x\in\Z$, by the above choice of $a_0,b_0,c_0$ and by the assumption that $m\ge5$, we have $0\le a_0+b_0+c_0<m$. This implies that
$n=\lfloor x^2/m\rfloor+\lfloor 0^2/m\rfloor+\lfloor 0^2/m\rfloor$.

In view of the above, we complete the proof.

\maketitle
\section{Proof of Theorem \ref{Thm a,b,c}}
\setcounter{lemma}{0}
\setcounter{theorem}{0}
\setcounter{corollary}{0}
\setcounter{remark}{0}
\setcounter{equation}{0}
\setcounter{conjecture}{0}

In this section, we adopt the notations as in Section 1 and Section 2, without loss of generality, we may assume that $\ord_2a\ge\ord_2b\ge\ord_2c$.

Before the proof of Theorem \ref{Thm a,b,c}, we set $e_i'=\frac12 e_i$ for $i=1,2,3$, and let
$L'=abc\langle a,b,c\rangle=\Z e_1'\perp\Z e_2'\perp\Z e_3'$. If $l_{a,b,c}(n)$ is represented by $L'_2+v$. Then there are $x_i\in\Z_2$ ($i=1,2,3$) such that
\begin{equation}\label{equation D}
l_{a,b,c}(n)=Q\left((x_1+\frac{\alpha}{a})e_1'+(x_2+\frac{\beta}{b})e_2'+(x_3+\frac{\gamma}{c})e_3'\right).
\end{equation}
For $i=1,2,3$, let $$\ve_i=\begin{cases}1&\mbox{if}\ 2\nmid x_i,\\0&\mbox{if}\ 2\mid x_i,\end{cases}$$
and let $\alpha'=\alpha+\ve_1a$, $\beta'=\beta+\ve_2b$, $\gamma'=\gamma+\ve_3c$. Then (\ref{equation D})
is equal to
\begin{equation*}
Q\left(\frac{x_1-\ve_1}{2}e_1+\frac{x_2-\ve_2}{2}e_2+\frac{x_3-\ve_3}{2}e_3+
\frac{1}{2}(\frac{\alpha'}{a}e_1+\frac{\beta'}{b}e_2+\frac{\gamma'}{c}e_3)\right).
\end{equation*}
If we replace $\alpha,\beta,\gamma$ by $\alpha',\beta',\gamma'$ respectively, then
it might worth mentioning here that $l_{a,b,c}(n)$ is invariant and the above results on the local representation of $l_{a,b,c}(n)$ over $\Z_p$ ($p\ne2$) are still valid.
Thus we only need to consider the representation of $l_{a,b,c}(n)$ by $L_2'+v$ over $\Z_2$.

\noindent{\it\bf Proof of Theorem 1.2}. Firstly, it is easy to see that $l_{a,b,c}(n)$ is represented by $L_2'+v$ if and only if
$l^*_{a,b,c}(n):=n+(a_0-\alpha^2)/a+(b_0-\beta^2)/b+(c_0-\gamma^2)/c$ is represented by $g^*(x,y,z):=ax^2+2\alpha x+by^2+2\beta y+cz^2+2\gamma z$
over $\Z_2$.
We will divide the proof into following three cases.

{\it Case} 1. $\ord_2a=1$ and $2\nmid bc$.

When $n\equiv 0\pmod4$, we put $\alpha=\beta=a_0=b_0=1$ and $c_0=\gamma^2$, where
$$\gamma=\begin{cases}1&\mbox{if}\ b\not\equiv c\pmod4,\\2&\mbox{if}\ b\equiv c\pmod4.\end{cases}$$
Then $l^*_{a,b,c}(n)=n$ and by (ii) of Lemma \ref{Lem local}, it is easy to see that
$l_{a,b,c}(n)$ can be represented by $L_2'+v$ and $l_{a,b,c}(n)\equiv 3\pmod 4$.

When $n\equiv 2\pmod4$, we put $\alpha'=a_0=1$, and $b_0=\beta'^2$, $c_0=\gamma'^2$, where
$$(\beta,\gamma)=\begin{cases}(1,2)&\mbox{if}\ b\not\equiv c\pmod4,\\(2,2)&\mbox{if}\ b\equiv c\pmod4.\end{cases}$$
Then $l^*_{a,b,c}(n)=n$. Since $n-b\cdot1^2-2\beta\cdot1-c\cdot1^2-4\cdot1\equiv 0\pmod4$,
by (ii) of Lemma \ref{Lem local}, it is easy to see that $l^*_{a,b,c}(n)=n$ can be represented by $g^*(x,y,z)$ over $\Z_2$. Hence
$l_{a,b,c}(n)$ can be represented by $L_2'+v$ and $l_{a,b,c}(n)\equiv 3\pmod 4$.

When $n$ is odd, put $\alpha=a_0=1$, $b_0=\beta^2$ and $c_0=\gamma^2$, where
$$\beta=\begin{cases}1&\mbox{if}\ b\not\equiv n\pmod4,\\2&\mbox{if}\ b\equiv n\pmod4,\end{cases}$$
and
$$\gamma=\begin{cases}\beta-1&\mbox{if}\ c\not\equiv b\pmod4,\\\beta&\mbox{if}\ c\equiv b\pmod4.\end{cases}$$
Since $n-b\cdot1^2-2\beta\cdot1\equiv 0\pmod4$, by (ii) of Lemma \ref{Lem local}, $l_{a,b,c}(n)$ can be represented by $L_2'+v$ and $l_{a,b,c}(n)\equiv 3\pmod 4$.

In sum, by our choice of $\alpha,\beta,\gamma$, $l_{a,b,c}(n)\equiv 3\pmod4$ and Lemma \ref{Lem local},
$l_{a,b,c}(n)$ can be represented by $L'+v$ locally. This implies that $l_{a,b,c}(n)$ can be represented by $L+v'$ locally, where
$v'=\frac12(\frac{\alpha'}{a}e_1+\frac{\beta'}{b}e_2+\frac{\gamma'}{c}e_3)$.
By the fact that $a,b,c$ are pairwisely coprime, we have any prime factor of $2abc$ is prime to
$l_{a,b,c}(n)$ and $l_{a,b,c}$ is not a square. Thus by (\ref{tr^2}), $\mathcal{S}_{a,b,c}=\emptyset$. Hence by Lemma \ref{Lem summary}, every sufficiently large integer $n$ can be written as
$$\lfloor x^2/a\rfloor+\lfloor y^2/b\rfloor+\lfloor z^2/c\rfloor,$$
with $x,y,z\in\Z$.

{\it Case} 2. $\ord_2a\ge2$ and $2\nmid bc$.

Note that in the proof of Case 1 we always put $\alpha=a_0=1$. Hence for any positive odd integers $a',b,c$ with
$b,c\ge5$, $\lfloor x^2/2a'\rfloor+\lfloor y^2/b\rfloor+\lfloor z^2/c\rfloor$ is almost universal. If $\ord_2a$ is
odd, in view of (\ref{tx2}), $\lfloor x^2/a\rfloor+\lfloor y^2/b\rfloor+\lfloor z^2/c\rfloor$ is almost universal.
With the same reason, when $\ord_2a$ is even, it is sufficient to prove the case when $\ord_2a=2$. Below we suppose that $\ord_2a=2$.

When $n$ is even, put $\alpha=\beta=a_0=b_0=1$ and $c_0=\gamma^2$, where
$$\gamma=\begin{cases}1&\mbox{if}\ bc\not\equiv1\pmod8,\\2&\mbox{if}\ bc\equiv1\pmod8.\end{cases}$$
Then by Lemma \ref{Lem local}, it is easy to see that the odd integer $l_{a,b,c}(n)$ is represented by $L_2'+v$ and
$l_{a,b,c}(n)\not\equiv 1\pmod8$.

When $n$ is odd, put $\alpha=\beta=a_0=b_0=1$, and $c_0=\gamma^2$, where
$$\gamma=\begin{cases}2&\mbox{if}\ bc\not\equiv1\pmod8,\\1&\mbox{if}\ bc\equiv1\pmod8.\end{cases}$$
Since $n-b\cdot1^2-2\beta\cdot1$ is even, by Lemma \ref{Lem local}, $l^*_{a,b,c}(n)$ can be represented by
$g^*(x,y,z)$ over $\Z_2$. This implies that 
the odd integer $l_{a,b,c}(n)$ is represented by $L_2'+v$ and
$l_{a,b,c}(n)\not\equiv 1\pmod8$.

In sum, by our choice of $\alpha,\beta,\gamma$ and Lemma \ref{Lem local}, the odd integer $l_{a,b,c}(n)$ can be represented by $L+v'$ locally and $l_{a,b,c}(n)\not\equiv1\pmod8$. By the fact that $a,b,c$ are pairwisely coprime and (\ref{tr^2}), $\mathcal{S}_{a,b,c}=\emptyset$. Thus by Lemma \ref{Lem summary}, $\mathcal{F}_{a,b,c}(x,y,z)$ is almost universal.

{\it Case} 3. $2\nmid abc$.

In this case, without loss of generality, we may assume that $b\equiv c\pmod4$. Since
$L'_2+v=L_2'$, by \cite[Lemma 5.2.6 and Propositon 5.2.3]{Ki}, we have
$$L_2'\cong\begin{cases}\begin{pmatrix} 2 & 1 \\1 & 2 \end{pmatrix}\perp \langle 3\rangle,&\mbox{if}\ a\equiv b\equiv c\pmod 4,\\\\\begin{pmatrix} 0 & 1 \\1 & 0 \end{pmatrix}\perp \langle -1\rangle,&\mbox{if}\ a\not\equiv b\equiv c\pmod4.\end{cases}$$

We first consider the case when $a\not\equiv b\equiv c\pmod 4$, then
$$L_2'\cong \begin{pmatrix} 0 & 1 \\1 & 0 \end{pmatrix}\perp \langle -1\rangle.$$ It is easy to see that $L_2'$ can represent all $2$-adic integers over $
\Z_2$.

When $n$ is even, let $\alpha=a_0=1$, and $b_0=c_0=\beta^2=\gamma^2$, where
$$\beta=\gamma=\begin{cases}1&\mbox{if}\ n\equiv0\pmod4,\\2&\mbox{if}\ n\equiv2\pmod4.\end{cases}$$

When $n$ is odd, put $\beta=b_0=1$ and $a_0=\alpha^2$, $c_0=\gamma^2$, where
$$(\alpha,\gamma)=\begin{cases}(2,1)&\mbox{if}\ a\equiv n\pmod4,\\(1,2)&\mbox{if}\ a\not\equiv n\pmod4.\end{cases}$$

In view of the above, by the choice of $\alpha,\beta,\gamma$ and Lemma \ref{Lem local},
we have $l_{a,b,c}(n)\equiv 3\pmod4$ and $l_{a,b,c}(n)$ is represented by $L+v'$ locally.
By the fact that $a,b,c$ are pairwisely coprime and (\ref{tr^2}), we have
$\mathcal{S}_{a,b,c}=\emptyset$. Thus by
Lemma \ref{Lem summary}, $\mathcal{F}_{a,b,c}(x,y,z)$ is almost universal.

Now, we consider the case when $a\equiv b\equiv c\pmod4$.
In this case, by the symmetry of $a,b,c$, without loss of generality, we may assume that $b\equiv c\pmod 8$, and we 
set $a+4\mu\equiv b\equiv c\pmod8$, where
$$\mu=\begin{cases}1&\mbox{if}\ a\not\equiv b\pmod8,\\0&\mbox{if}\ a\equiv b\pmod8.\end{cases}$$

Reall that
$$L_2'\cong\begin{pmatrix} 2 & 1 \\1 & 2 \end{pmatrix}\perp \langle 3\rangle,$$
and $$Q\left(\begin{pmatrix} 2 & 1 \\1 & 2 \end{pmatrix}\right)=\{x\in\Z_2: \ord_2x\equiv1\pmod2\}\cup\{0\}.$$
It is easy to see that
$$\{4m+1,8m+3,8m+6: m\in\Z_2\}\subseteq Q(L_2').$$
We further set
$$\mathcal{A}=\{l_{a,b,c}(n)+8\Z: a_0,b_0,c_0=1,4\}.$$
Via computation, we have 
\begin{equation}\label{equation E}
\{cab_0+abc_0+8\Z: b_0,c_0=1,4\}=\{0+8\Z,2+8\Z,(5-4\mu)+8\Z\}.
\end{equation}

When $n\equiv0\pmod4$, let $a_0=\alpha^2$, $b_0=\beta^2$, and $c_0=\gamma^2$, where
$$(\alpha,\beta,\gamma)=\begin{cases}(1,2,2)&\mbox{if}\ n\equiv4\pmod8,\\(1,1,1)&\mbox{if}\ n\equiv0\pmod8.\end{cases}$$
Then 
$$l_{a,b,c}(n)\equiv\begin{cases}5\pmod8&\mbox{if}\ n\equiv4\pmod8,\\3\pmod8&\mbox{if}\ n\equiv0\pmod8.\end{cases}$$ 
It is clear that $l_{a,b,c}(n)$ is represented by $L_2'$.

When $n\equiv1\pmod4$, suppose first that $n\equiv5\pmod8$. Let $a_0=\alpha^2$, where
$$\alpha=\begin{cases}1&\mbox{if}\ a\equiv 1,7\pmod8,\\2&\mbox{if}\ a\equiv 3,5\pmod8.\end{cases}$$
Then $abcn+bca_0\equiv 3,4,5,6\pmod8$. Thus by (\ref{equation E}), $3+8\Z,\ 5+8\Z,\ 6+8\Z$ are all in $\mathcal{A}$. Therefore, there exist $\alpha,\beta,\gamma\in\{1,2\}$ such that $l_{a,b,c}(n)$ can be represented by $L_2'$ and
$l_{a,b,c}(n)+8\Z\in\{3+8\Z,5+8\Z,6+8\Z\}$.

Suppose now that $n\equiv 1\pmod 8$. Let
$$\alpha=\begin{cases}1&\mbox{if}\ a\equiv 3,5\pmod8,\\2&\mbox{if}\ a\equiv 1,7\pmod8.\end{cases}$$
Then $abcn+bca_0\equiv3,4,5,6\pmod8$. By (\ref{equation E}), $3+8\Z,5+8\Z,6+8\Z$ are all in $\mathcal{A}$.
Hence there exist $\alpha,\beta,\gamma\in\{1,2\}$ such that $l_{a,b,c}(n)$ can be represented by $L_2'$ and
$l_{a,b,c}(n)+8\Z\in\{3+8\Z,5+8\Z,6+8\Z\}$.

When $n\equiv 3\pmod4$, suppose first that $n\equiv7\pmod8$. Let $a_0=\alpha^2$, where
$$\alpha=\begin{cases}1&\mbox{if}\ a\equiv 3,5\pmod8,\\2&\mbox{if}\ a\equiv 1,7\pmod8.\end{cases}$$
Then $abcn+bca_0\equiv3,4,5,6\pmod8$. By (\ref{equation E}), $3+8\Z,5+8\Z,6+8\Z$ are all in $\mathcal{A}$. 
Therefore, there exist $\alpha,\beta,\gamma\in\{1,2\}$ such that $l_{a,b,c}(n)$ can be represented by $L_2'$ and
$l_{a,b,c}(n)+8\Z\in\{3+8\Z,5+8\Z,6+8\Z\}$.

Suppose now that $n\equiv 3\pmod 8$. Let
$$\alpha=\begin{cases}1&\mbox{if}\ a\equiv 1,7\pmod8,\\2&\mbox{if}\ a\equiv 3,5\pmod8.\end{cases}$$
Then $abcn+bca_0\equiv3,4,5,6\pmod8$. By (\ref{equation E}), $3+8\Z,5+8\Z,6+8\Z$ are all in $\mathcal{A}$.
Hence there exist $\alpha,\beta,\gamma\in\{1,2\}$ such that $l_{a,b,c}(n)$ can be represented by $L_2'$ and
$l_{a,b,c}(n)+8\Z\in\{3+8\Z,5+8\Z,6+8\Z\}$.

When $n\equiv 2\pmod4$, suppose first that $n\equiv6\pmod8$ and $\mu=1$. When this occurs, we set $a_0=\alpha^2$, where
$$\alpha=\begin{cases}1&\mbox{if}\ a\equiv 3\pmod4,\\2&\mbox{if}\ a\equiv 1\pmod4.\end{cases}$$
Then $abcn+bca_0\equiv 2,3\pmod 8$. Note that $(\ref{equation E})=\{0+8\Z,1+8\Z,2+8\Z\}$. Thus
$3+8\Z,5+8\Z$ are in $\mathcal{A}$. Hence there exist $\alpha,\beta,\gamma\in\{1,2\}$ such that $l_{a,b,c}(n)$ can be represented by $L_2'$ and
$l_{a,b,c}(n)+8\Z\in\{3+8\Z,5+8\Z\}$.

Suppose now that $n\equiv 6\pmod8$ and $\mu=0$. If $a\equiv3\pmod4$, we put $\alpha=a_0=1$,
$b_0=\beta^2$ and $c_0=\gamma^2$. Then
$abcn+bca_0\equiv 3\pmod8$. By (\ref{equation E}), $5+8\Z\in\mathcal{A}$. Hence there exist $\alpha,\beta,\gamma\in\{1,2\}$ such that $l_{a,b,c}(n)$ can be represented by $L_2'$ and $l_{a,b,c}(n)\equiv5\pmod8$.

If $a\equiv1\pmod4$, we set $a=4l+1$ with $l\in\Z$. Since Sun \cite{Sun15} and S. Mezroui, A. Azizi and M. Ziane \cite{MAZ}
proved that $\mathcal{F}_5$ and $\mathcal{F}_9$ can represent all natural numbers respectively. By the symmetry of $a,b,c$ in this case (note that $a\equiv b\equiv c\pmod8$), we may assume that $l\ge3$. If we put $\alpha=2l+1$, then one may easily verify that $a_0=3l+1$. If we put $\alpha=2l+2$, then $a_0=3l+3<4l+1$. When $l$ is even, then one of $abcn+3l+1$, $abcn+3l+3$ is not
congruent to $7$ modulo $8$. According to (\ref{equation E}), there exist $\beta,\gamma\in\{1,2\}$ and
$\alpha\in\{2l+1,2l+2\}$ such that $l_{a,b,c}(n)$ is congruent to $3$ or $5$ modulo $8$. When $l$ is odd, then one of $abcn+3l+1$, $abcn+3l+3$ is not congruent to $2$ modulo $8$. By (\ref{equation E}), there are  $\beta,\gamma\in\{1,2\}$ and $\alpha\in\{2l+1,2l+2\}$ such that $l_{a,b,c}(n)$ is congruent to $5$ or $6$ modulo $8$. It might worth mentioning here
that $(a,2l+1)=1$ and that $(a,2l+2)\in\{1,3\}$. If $(a,2l+2)=3$, then $a=4l+1\equiv0\pmod3$.
Thus $3\nmid bc\beta\gamma$ and $3\mid l_{a,b,c}(n)$. By Hensel's Lemma, one may easily verify that
$l_{a,b,c}(n)$ can be represented by $L_3+v$. In view of the above, by Lemma \ref{Lem local}, $l_{a,b,c}(n)$ can be represented by
$L'+v$ locally and $l_{a,b,c}(n)$ is congruent to $3,5$ or $6$ modulo $8$.

Now, we consider the case when $n\equiv 2\pmod8$. Suppose first that $\mu=1$. Let
$$\alpha=\begin{cases}1&\mbox{if}\ a\equiv 1\pmod4,\\2&\mbox{if}\ a\equiv 3\pmod4.\end{cases}$$
Then $abcn+bca_0\equiv 2,3\pmod8$. By (\ref{equation E}), $3+8\Z\in\mathcal{A}$. Hence there exist $\alpha,\beta,\gamma\in\{1,2\}$ such that $l_{a,b,c}(n)$ can be represented by $L_2'$ and
$l_{a,b,c}(n)\equiv3\pmod8$.

Now, suppose that $\mu=0$. If $a\equiv1\pmod4$, we let $\alpha=1$. By (\ref{equation E}), it is easy to see that
$3+8\Z\in\mathcal{A}$. Hence there exist $\alpha,\beta,\gamma\in\{1,2\}$ such that $l_{a,b,c}(n)$ can be represented by $L_2'$ and $l_{a,b,c}(n)\equiv3\pmod8$.
If $a\equiv 3\pmod4$,
we set $a=4l+3$ with $l\ge1$. If we put $\alpha=2l+1$, then one may easily verify that $a_0=l+1$. If we put $\alpha=2l$, then $a_0=l+3$. When $l$ is even, one of $abcn+l+1$, $abcn+l+3$ is not congruent to $7$ modulo $8$. When $l$ is odd, one of $abcn+l+1$, $abcn+l+3$ is not congruent to $2$ modulo $8$. It might worth mentioning here
that $(a,2l+1)=1$ and that $(a,2l)\in\{1,3\}$. If $(a,2l)=3$, with the same reason as the above, by Hensel's Lemma, one may easily verify that
$l_{a,b,c}(n)$ can be represented by $L_3+v$. In view of the above, by Lemma \ref{Lem local} and (\ref{equation E}),
 there exist $\beta,\gamma\in\{1,2\}$ and $\alpha\in\{2l+1,2l\}$ such that $l_{a,b,c}(n)$ can be
 represented by $L'+v$ locally and
$l_{a,b,c}(n)+8\Z\in\{3+8\Z,5+8\Z,6+8\Z\}$.

In sum, by our choice of $\alpha,\beta,\gamma$ and Lemma \ref{Lem local},
$l_{a,b,c}(n)$ is represented by $L+v'$ locally. By the fact that $a,b,c$ are pairwisely coprime, any odd prime divisor of $abc$ is prime to $l_{a,b,c}(n)$.
Note also that $l_{a,b,c}(n)+8\Z\in\{3+8\Z,5+8\Z,6+8\Z\}$, thus $l_{a,b,c}(n)$ is neither a square nor two times a square. Hence by (\ref{tr^2}), we have $\mathcal{S}_{a,b,c}=\emptyset$. By Lemma \ref{Lem summary}, every sufficiently large integer $n$ can be written as
$$\lfloor x^2/a\rfloor+\lfloor y^2/b\rfloor+\lfloor z^2/c\rfloor,$$
with $x,y,z\in\Z$.

The proof of Theorem \ref{Thm a,b,c} is complete.\qed

\acknowledgment We would like to thank Prof. Zhi-Wei Sun for his helpful comments and steadfast encouragement.

This research was supported by the National Natural Science Foundation of
China (Grant No. 11571162).

\end{document}